\documentclass[reqno,12pt,a4paper]{amsart}

\usepackage[utf8]{inputenc}

\usepackage{enumerate}
\usepackage{enumitem}
\usepackage{amstext,amsmath,amsthm,amsfonts,amssymb,amscd}
\usepackage{latexsym,mathrsfs,dsfont,euscript}
\usepackage{mathbbol}
\usepackage{mathtools}
\usepackage{xcolor}
\usepackage{multicol}
\usepackage{fullpage}

\usepackage{hyperref} 
\hypersetup{
	colorlinks,%
	citecolor=blue,%
	filecolor=red,%
	linkcolor=red,%
	urlcolor=blue
}

\usepackage{pgfplots}
\pgfplotsset{compat=1.15}
\usepackage{mathrsfs}
\usetikzlibrary{arrows}

\usepackage[font=footnotesize,labelfont=it]{caption}

\newtheorem{theorem}{Theorem}[section]
\newtheorem{proposition}[theorem]{Proposition}
\newtheorem{lemma}[theorem]{Lemma}

\theoremstyle{definition}

\theoremstyle{remark}

\numberwithin{equation}{section}

\allowdisplaybreaks

\begin{document}
	
\title[Boundedness properties of the bilinear fractional integral operators induced by hypermetrics of third order]{Boundedness properties of the bilinear fractional integral operators induced by hypermetrics of third order}


\author[]{Hugo Aimar}
\author[]{Ivana G\'{o}mez}
\author[]{Joaqu\'in Toledo}



%
%
\begin{abstract}
We introduce a natural bilinear fractional integral type operator induced by a third order hypermetric on Ahlfors regular quasi-metric spaces. Given a quasi-metric space $(X,d)$ the function $\rho(x,y,z)$, defined as the distance, in $X^3$, of $(x,y,z)$ to the diagonal $\bigtriangleup_3=\{(x,x,x)\in X^3:x\in X\}$ is said to be a third order hypermetric in $X$. When $(X,d)$ is a Euclidean space or, more generally, when $(X,d,\mu)$ is $\eta$-Ahlfors regular for some $\eta$ positive, the function $\rho(x,y,z)$ generates kernels for bilinear operators of the type $T^{\gamma}(f,g)(x)=\iint_{X\times X}\rho(x,y,z)^{-\gamma}f(y)g(z)d\mu(y)d\mu(z)$, for a given positive $\gamma$. In the setting of $\eta$-Ahlfors regular space, the power $-\gamma=-2\eta$ of $\rho(x,\cdot,\cdot)$ provides the natural singularity for this family of kernels. In this paper we consider the fractional integral rank $0<\gamma<2\eta$. We prove boundedness properties of the type $\|T^{\gamma}(f,g)\|_{p_3}\leq C\|f\|_{p_1}\|g\|_{p_2}$ for  adequate values of the exponents $p_1,p_2$ and $p_3$. The proof is based on three upper bounds for $T^{\gamma}(f,g)$ in terms of the classical linear fractional Riesz operators $I_{\eta-\frac{\gamma}{2}}$, using the linear Hardy-Littlewood-Sobolev inequality.
\end{abstract}

\maketitle

\section{Introduction}\label{sec:introduction}
The classical linear fractional integral operator $I_{\alpha}$, in particular its boundedness pro\-perties on Lebesgue space, has been extended to the multilinear setting. See for example \cite{Gra92}, \cite{KeSte99}, \cite{GrafakosBook2014}. On the other hand the boundedness of the linear Riesz operator $I_{\alpha}$ in Lebesgue spaces, that in the Euclidean setting can be found for example in \cite{Stein1970book}, was extended to metric measure spaces in \cite{GCGa04}. The main results in all these papers provide extensions of the Hardy-Littlewood-Sobolev inequality.

In \cite{AiGoTo2026MultiApproxIdent} we introduced the notion of hypermetric of order $k$ on a quasi-metric space. The basic idea is provided by the fact that given two different points $x$ and $y$ in the metric space $(X,d)$ we have that $d(x,y)$ is equivalent to $\rho(x,y)$ given as the distance of the pair $(x,y)\in X\times X$ to the diagonal $\bigtriangleup_2$ of $X\times X$ with respect to the metric $d^{(2)}\left((x,y); (x',y')\right)=\max\{d(x,x'),d(y,y')\}$. In fact, from the triangle inequality for $d$ we easily see that $\frac{1}{2}d(x,y)\leq \rho(x,y)\leq d(x,y)$. When instead of a metric space we have a quasi-metric space $(X,d)$ with triangular constant $\kappa\geq 1$, we have that $\frac{1}{2\kappa}d(x,y)\leq \rho(x,y)\leq d(x,y)$ for every $x, y \in X$. For a given quasi-metric space $(X,d)$, we consider the product space $(X^{3},d^{(3)})$ with $X^3=X\times X\times X$ and $d^{(3)}(\boldsymbol{x},\boldsymbol{y})=\sup\{d(x_i,y_i), i=0, 1,2\}$, where $\boldsymbol{x}=(x_0,x_1, x_2)$ and $\boldsymbol{y}=(y_0,y_1, y_2)$. Set $\bigtriangleup_{3}=\{(x, x, x)\in X^{3}: x\in X\}$ to denote the diagonal of $X^{3}$. The \textbf{third order hypermetric} induced by $d$ in $X$ is defined by $\rho(\boldsymbol{x})=\rho(x_0,x_1, x_2)=d^{(3)}(\boldsymbol{x},\bigtriangleup_{3})$. We shall use $(x,y,z)$ to denote the points in $X^3$. Hence $\rho(x,y,z)=d^{(3)}\left((x,y,z),\bigtriangleup_3\right)=\inf_{u\in X}d^{(3)}\left((x,y,z), (u,u,u)\right)$.The function $\rho$ defined in $X^3$ generates a wide family of kernels for bilinear operators of the type $\varphi\left(\rho(x,y,z)\right)$.

Recall that for $\eta>0$ given, a space of homogeneous type $(X,d,\mu)$ is said to be $\eta$-Ahlfors regular, if there exist two constants $0<a\leq A<\infty$ such that the inequalities $ar^{\eta}\leq \mu\left(B_{d}(x,r)\right)\leq Ar^{\eta}$ hold for every $x\in X$ and every $r>0$. Here $B_{d}(x,r)=\{y\in X: d(x,y)<r\}$ is the $d$-ball centered at $x$ with radius $r>0$. The bilinear integral operators induced by $\rho$ and $\varphi$ have the general form
\begin{align*}
	T_{\varphi}(f,g)(x)=\iint_{X\times X}\varphi(\rho(x,y,z))f(y)g(z)d\mu(y)d\mu(z)
\end{align*}
under adequate conditions on $\varphi$, $f$ and $g$ that guarantee the existence of the integral.

The next result provides a necessary and sufficient condition on a nonincreasing, nonnegative function $\varphi(t)$ of the positive variable $t$ in order to have the convergence of the integral $\iint_{X\times X}\varphi(\rho(x,y,z))d\mu(y)d\mu(z)$ on an $\eta$-Ahlfors regular space.

\begin{lemma}\label{C1varphi<I<C2varphi}
Let $(X,d, \mu)$ be an $\eta$-Ahlfors regular space. Let $\varphi:\mathbb{R}_{>0} \to \mathbb{R}_{\geq 0}$ be a nonincreasing function. Then, there exist two constants $0<C_1\leq C_2<\infty$ depending only on the geometric constants $\kappa$, $a$ and $A$ such that the inequalities 
\begin{align*}
	C_1\int_0^\infty\varphi(t)t^{2\eta-1}dt\leq \iint_{X\times X}\varphi(\rho(x,y,z))d\mu(y)d\mu(z)\leq C_2\int_0^\infty\varphi(t)t^{2\eta-1}dt
\end{align*}
hold for every $x \in X$.
\end{lemma}
We shall prove the above result in Section~\ref{sec: Proof of Lemma 1.1 }.
Taking, $\varphi(t)=t^{-\alpha}\chi_{(0,1)}(t)$, $\alpha>0$, we see that $ \iint_{\{(y,z):\rho(x,y,z)<1\}}\rho^{-\alpha}(x,y,z)d\mu(y)d\mu(z)$ is finite if and only if $0<\alpha<2\eta$. On the other hand, taking $\varphi(t)=\chi_{(0,1)}(t)+t^{-\beta}\chi_{[1,\infty)}(t)$ in Lemma~\ref{C1varphi<I<C2varphi}, we also see that  $ \iint_{\{(y,z):\rho(x,y,z)\geq1\}}\rho^{-\beta}(x,y,z)d\mu(y)d\mu(z)$ is finite if and only if $\beta>2\eta$. Hence the power $2\eta$ of $\frac{1}{\rho(x,y,z)}$, for $x \in X$, determines the singularity of this family of kernels. Thus the natural fractional integral bilinear operator induced by the hypermetric $\rho$ is given by 
	\begin{align*}
	T^{\gamma}(f,g)(x)=\iint_{X\times X}\frac{f(y)g(z)}{\rho^{\gamma}(x,y,z)}d\mu(y)d\mu(z)
\end{align*}
with $0<\gamma<2\eta$.

The main result of this paper, that we shall prove in Section~\ref{sec: Proof of Theo 1.2}, concerning the bounded\-ness in Lebesgue spaces of $T^{\gamma}$ is the following. 
\begin{theorem}\label{|T|p3leqC|f|p1|g|p2}
Let $(X,d, \mu)$ be an $\eta$-Ahlfors regular space with $\eta>0$, and $\rho$ as before. Let $0<\gamma<2\eta$. Then, for every $p_1>1$, $p_2>1$ and $p_3$ such that $0<\frac{1}{p_3}=\frac{1}{p_1} + \frac{1}{p_2}-\frac{2\eta-\gamma}{\eta}$, there exists $C>0$ such that the inequality 
\begin{align*}
	\|T^{\gamma}(f,g)\|_{p_3}\leq C \|f\|_{p_1} \|g\|_{p_2}
\end{align*}
holds for every couple $f$, $g$ of measurable nonnegative functions.
\end{theorem}

\section{Proof of Lemma \ref{C1varphi<I<C2varphi}}\label{sec: Proof of Lemma 1.1 }
The results of this section are contained in \cite{AiGoTo2026MultiApproxIdent}, we include them for the sake of completeness.
The proof of Lemma~\ref{C1varphi<I<C2varphi} is based on the following control of the sections at $x \in X$, $E(x,r)=\{(y,z)\in X\times X:\rho(x,y,z)<r \}$, of the neighborhood of $\bigtriangleup_3$ given by $\{(x,y,z): \rho(x,y,z)<r\}$, $r>0$. With $\mu^2$ we shall denote the product measure $\mu \times \mu$ on $X\times X$.

\begin{lemma}\label{lemadeinclusiones}
Let $(X,d,\mu)$ be an $\eta$-Ahlfors regular space with geometric constants $\kappa$, $a$ and $A$. Let $\rho$ be the hypermetric of third order induced by $d$. Then, for every $x\in X$ and every $r>0$, we have,
\begin{enumerate}[label=\textit{(\thetheorem.\alph*)}, leftmargin=*, align=left, font=\color{black}]
\item \label{incluBenE}
$B_d(x,r)\times B_d(x,r)\subset E(x,r) \subset	B_d(x,2\kappa r)\times B_d(x,2\kappa r)$, and
\item\label{AhlforsconE}
$a^{2}r^{2\eta}\leq \mu^2\left(E(x,r)\right)\leq (2\kappa)^{2\eta}A^{2}r^{2\eta}$.
\end{enumerate}
\end{lemma}
\begin{proof}
	Notice that \ref{AhlforsconE} follows from \ref{incluBenE} and the $\eta$-Ahlfors character of $(X,d,\mu)$. The first inclusion in \ref{incluBenE} follows from the fact that if $(y,z)\in B_d(x,r)\times B_d(x,r)$, then $\rho(x,y,z)\leq \max\{d(x,x), d(x,y), d(x,z)\}<r$. Hence $(y,z) \in E(x,r)$. To prove the second inclusion in \ref{incluBenE}, take now  $(y,z)\in E(x,r)$. Then $\rho(x,y,z)<r$. So that, there exists $u\in X$ such that, $ \max\{d(x,u), d(y,u), d(z,u)\}<r$. Hence $d(y,x)\leq \kappa(d(y,u)+d(u,x))<2\kappa r$ and $d(z,x)\leq \kappa(d(z,u)+d(u,x))<2\kappa r$, and $(y,z)\in B_d(x,2\kappa r)\times B_d(x,2\kappa r)$.
\end{proof}

\begin{proof}[Proof of Lemma~\ref{lemadeinclusiones}]
Set
\begin{align*}
J(x):=	\iint_{X\times X}\varphi(\rho(x,y,z))d\mu(y)d\mu(z).
\end{align*}
Take $x\in X$ and $r>0$, then, since $\mu^2(\{(x,x)\})=0$, for $\lambda>1$
\begin{align*}
	J(x)=\sum_{j\in \mathbb{Z}}\iint_{ \lambda^{j}\leq \rho(x,y,z)< \lambda^{j+1}}\varphi(\rho(x,y,z))d\mu(y)d\mu(z). 
\end{align*}
Hence from the nonincreasing condition on $\varphi$ and \ref{incluBenE} in Lemma~\ref{lemadeinclusiones},
\begin{align}
	J(x) &\leq \sum_{j\in \mathbb{Z}}\varphi\left(\lambda^{j}\right)\iint_{ \lambda^{j}\leq \rho(x,y,z)< \lambda^{j+1}}d\mu(y)d\mu(z) \notag\\ \notag
	&\leq \sum_{j\in \mathbb{Z}}\varphi\left(\lambda^{j}\right) \mu^{2}\left(E(x, \lambda^{j+1})\right)\\ \notag
	&\leq A^{2}(2\kappa)^{2\eta }\sum_{j\in \mathbb{Z}}\varphi\left(\lambda^{j}\right)\lambda^{2\eta(j+1)}\\ \notag
	&=\frac{\lambda^{2\eta}A^2 (2\kappa)^{2\eta }}{\log \lambda}\sum_{j\in \mathbb{Z}}\int_{\lambda^{j-1}}^{\lambda^{j}}\varphi\left(\lambda^{j}\right) \lambda^{2\eta j}\frac{dt}{t}\\ \notag
	&\leq \frac{\lambda^{4\eta}A^2 (2\kappa)^{2\eta }}{\log \lambda}\sum_{j\in \mathbb{Z}}\int_{\lambda^{j-1}}^{\lambda^{j}}\varphi\left(t\right) t^{2\eta-1}dt\\ 
	&=\left(\frac{\lambda^{4\eta}}{\log \lambda}\right)A^2 (2\kappa)^{2\eta} \int_0^\infty\varphi(t)t^{2\eta-1}dt.\label{eq:upperboundJ}
\end{align}
Let us consider now the lower bound for $J(x)$.		
From Lemma~\ref{lemadeinclusiones} we have that for $\lambda>2\kappa$,
\begin{equation*}
	E(x,\lambda^{j+1})\setminus E(x,\lambda^{j})\supset  [B_d(x,\lambda^{j+1})\times B_d(x,\lambda^{j+1})]\setminus [B_d(x,2\kappa \lambda^{j}) \times B_d(x,2\kappa \lambda^{j})].
\end{equation*}
Now, 
\begin{align*}
	J(x)&\geq \sum_{j\in \mathbb{Z}}\varphi\left(\lambda^{j+1}\right)\mu^2\left(E(x,\lambda^{j+1})\setminus E(x,\lambda^{j}) \right)\\
	& \geq \sum_{j\in \mathbb{Z}}\varphi\left(\lambda^{j+1}\right)\mu^2\left([B_d(x,\lambda^{j+1})\times B_d(x,\lambda^{j+1})]\setminus  [B_d(x,2\kappa \lambda^{j})\times B_d(x,2\kappa \lambda^{j})]\right).
\end{align*}
From the Ahlfors character of the space,
\begin{align*}
	\mu^2 &\left([B_d(x,\lambda^{j+1})\times B_d(x,\lambda^{j+1})]\setminus  [B_d(x,2\kappa \lambda^{j})\times B_d(x,2\kappa \lambda^{j})]\right)\\
	&\phantom{ B_d(x,\lambda^{j+1})} =\mu^2\left( B_d(x,\lambda^{j+1}) \times B_d(x,\lambda^{j+1})\right)- \mu^2\left( B_d(x,2\kappa \lambda^{j}) \times B_d(x,2\kappa \lambda^{j})\right)\\
	&\phantom{ B_d(x,\lambda^{j+1})} \geq a^2 (\lambda^{j+1})^{2\eta}-A^2(2\kappa \lambda^{j})^{2\eta}\\
	&\phantom{ B_d(x,\lambda^{j+1})} =\lambda^{2\eta j}\left[a^2\lambda^{2\eta}-A^2 2^{2\eta}\kappa^{2\eta}\right].
\end{align*}
Hence, if we choose $\lambda=\dfrac{(1+A^{2}(2\kappa)^{2\eta})^{\tfrac{1}{2\eta}}}{a^{\tfrac{1}{\eta}}}$, we get
\begin{align*}
	J(x)&\geq \sum_{j\in \mathbb{Z}}\varphi(\lambda^{j+1})\lambda^{2\eta j}\\
	&\geq \frac{1}{(\lambda^{2\eta})^2}\frac{1}{\log \lambda}\sum_{j\in \mathbb{Z}}\int_{\lambda^{j+1}}^{\lambda^{j+2}}\varphi(t)t^{2\eta} \frac{dt}{t}\\
	&\geq \frac{1}{\lambda^{4\eta}\log \lambda} \int_0^\infty\varphi(t)t^{2\eta-1}dt. 
\end{align*}
This inequality together with \eqref{eq:upperboundJ}, gives the result.
\end{proof}
	
\section{Proof of Theorem \ref{|T|p3leqC|f|p1|g|p2}}\label{sec: Proof of Theo 1.2}
The next elementary geometric observation will be useful at describing the type regions for the operator $T^{\gamma}$.

\begin{lemma}\label{Omega=AuBuC}
Let $0<\sigma<1$ be given and
\begin{equation*}
	\Omega_\sigma = \{(r,s)\in (0,1)^2: r+s>2\sigma\}.
\end{equation*}
Set $A_\sigma=(\sigma,1)^2$, $B_\sigma=\{(r,s)\in (0,1)^2:s>\sigma \text{ and } 2\sigma<r+s<1+ \sigma\}$ and $C_\sigma=\{(r,s)\in (0,1)^2:r>\sigma \text{ and } 2\sigma<r+s<1+ \sigma\}$ (see Figure~\ref{fig1}).
Then 
\begin{equation*}
	\Omega_\sigma = A_\sigma\cup B_\sigma \cup C_\sigma.
\end{equation*}
\end{lemma}
\begin{proof}
Notice first that $A_\sigma\cup B_\sigma\cup C_\sigma \subset \Omega_\sigma$. Let us now prove that $\Omega_\sigma$ is covered by $A_\sigma\cup B_\sigma\cup C_\sigma$. In fact, if $(r,s) \notin (\sigma,1)^2$ but $(r,s)\in (0,1)^2$ and $r+s>2\sigma$, then $0<r\leq \sigma$ or $0<s\leq \sigma$.
Assume that $0<r\leq \sigma$. Hence since $r+s>2\sigma$, we necessarily have that $s>\sigma$. Also $r+s\leq\sigma+s<\sigma+1$. Hence $(r,s)\in B_\sigma$. In case $0<s\leq \sigma$, we see in the same way that $(r,s)\in C_\sigma$.
\end{proof}

\begin{figure}[hb]
\begin{multicols}{4}
	\begin{center}
		\includegraphics[width=0.36\textwidth]{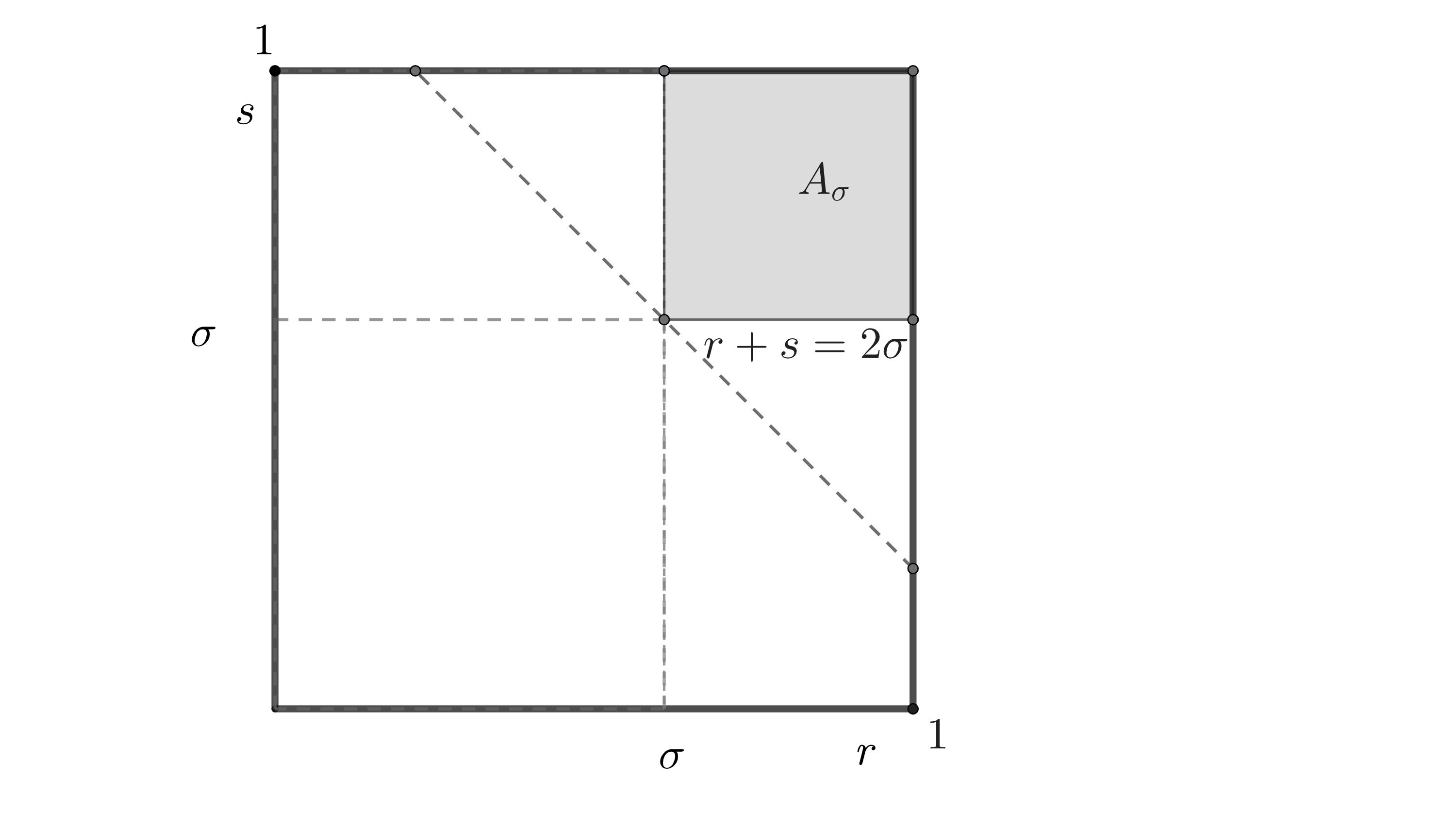}
	\end{center}
	\columnbreak
	\begin{center}
		\includegraphics[width=0.36\textwidth]{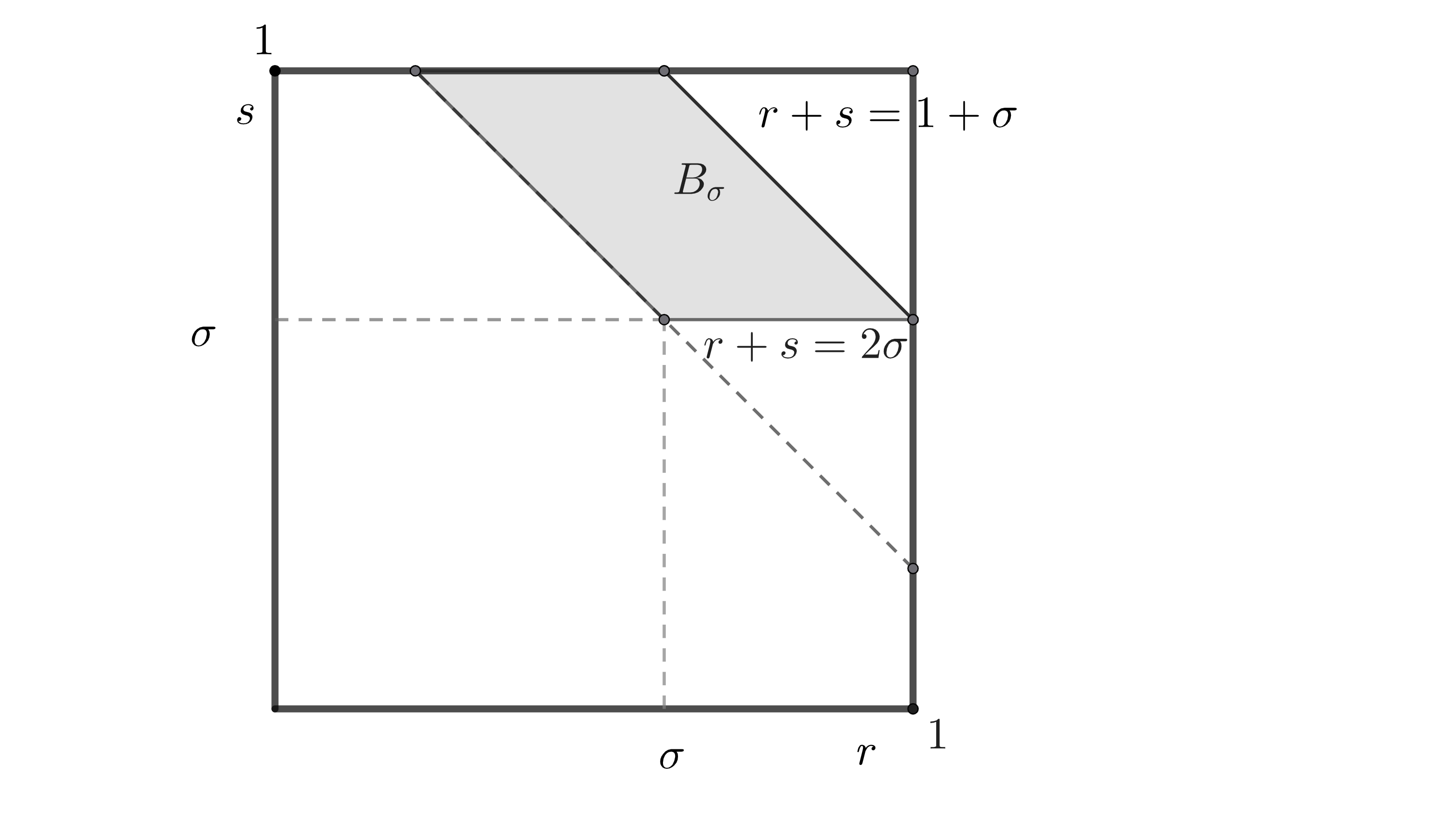}
	\end{center}
	\columnbreak
	\begin{center}
		\includegraphics[width=0.36\textwidth]{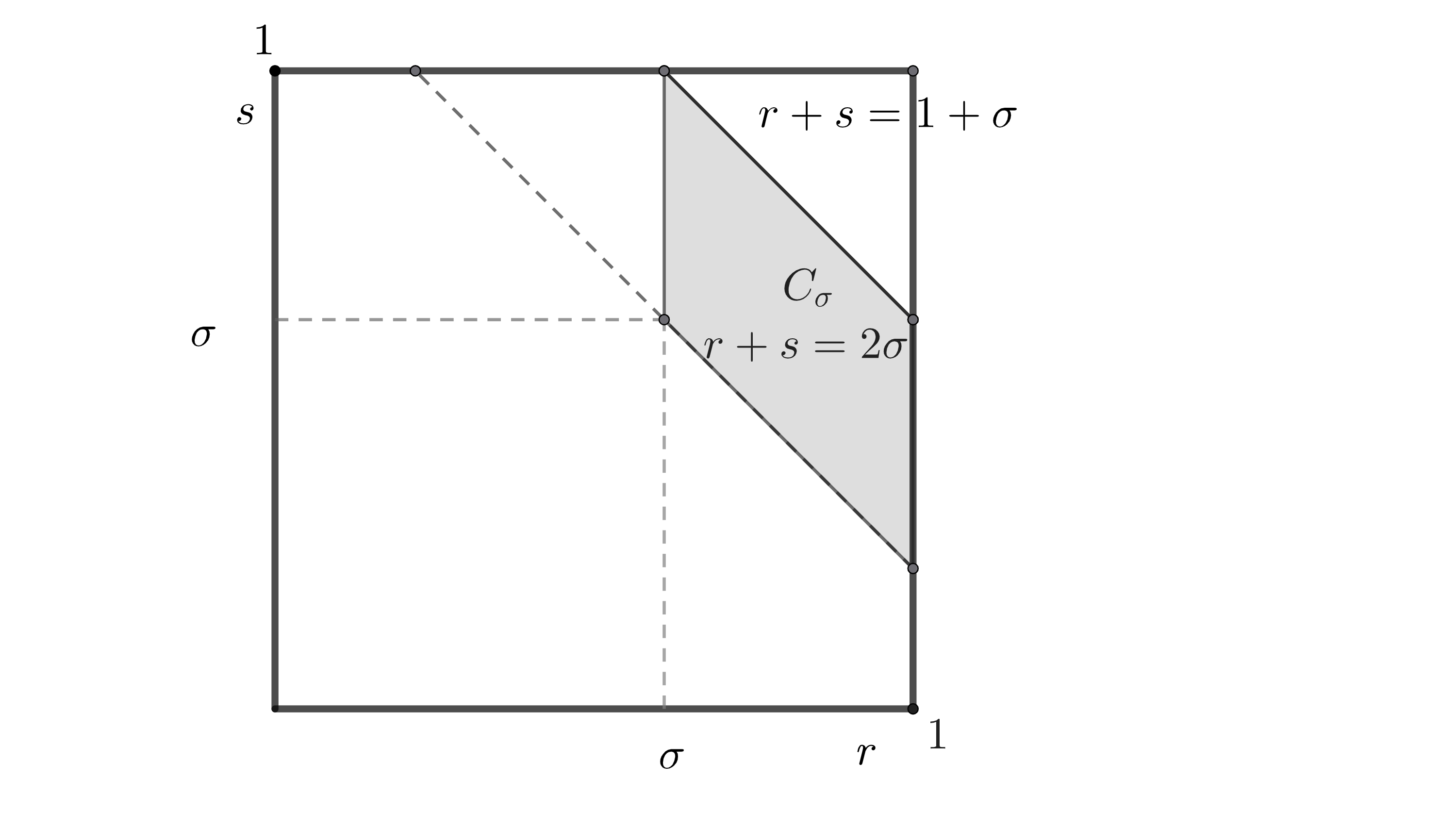}
	\end{center}
		\columnbreak
	\begin{center}
		\includegraphics[width=0.36\textwidth]{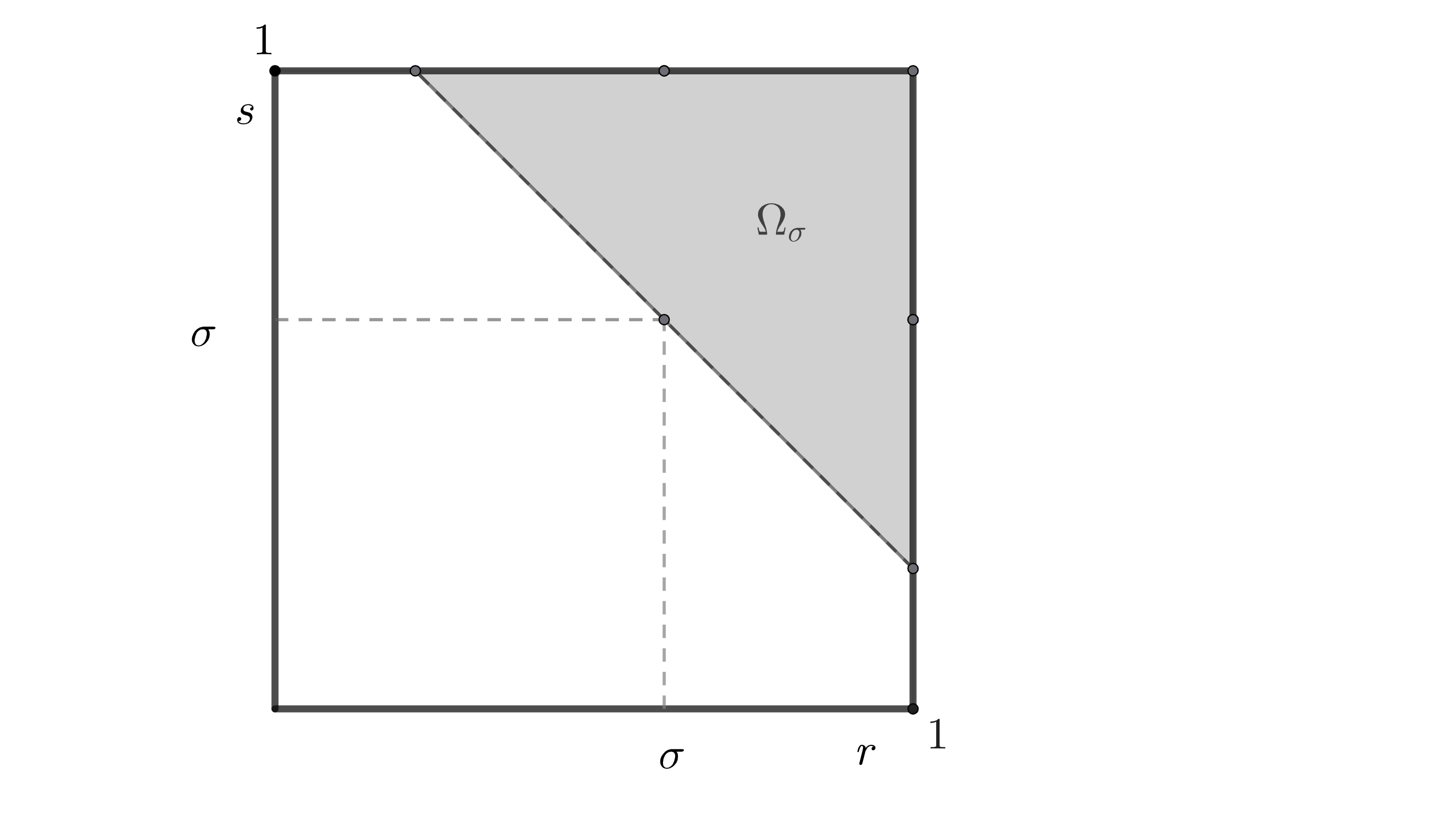}
	\end{center}
\end{multicols}
\begin{multicols}{4}
	\begin{center}
		\includegraphics[width=0.36\textwidth]{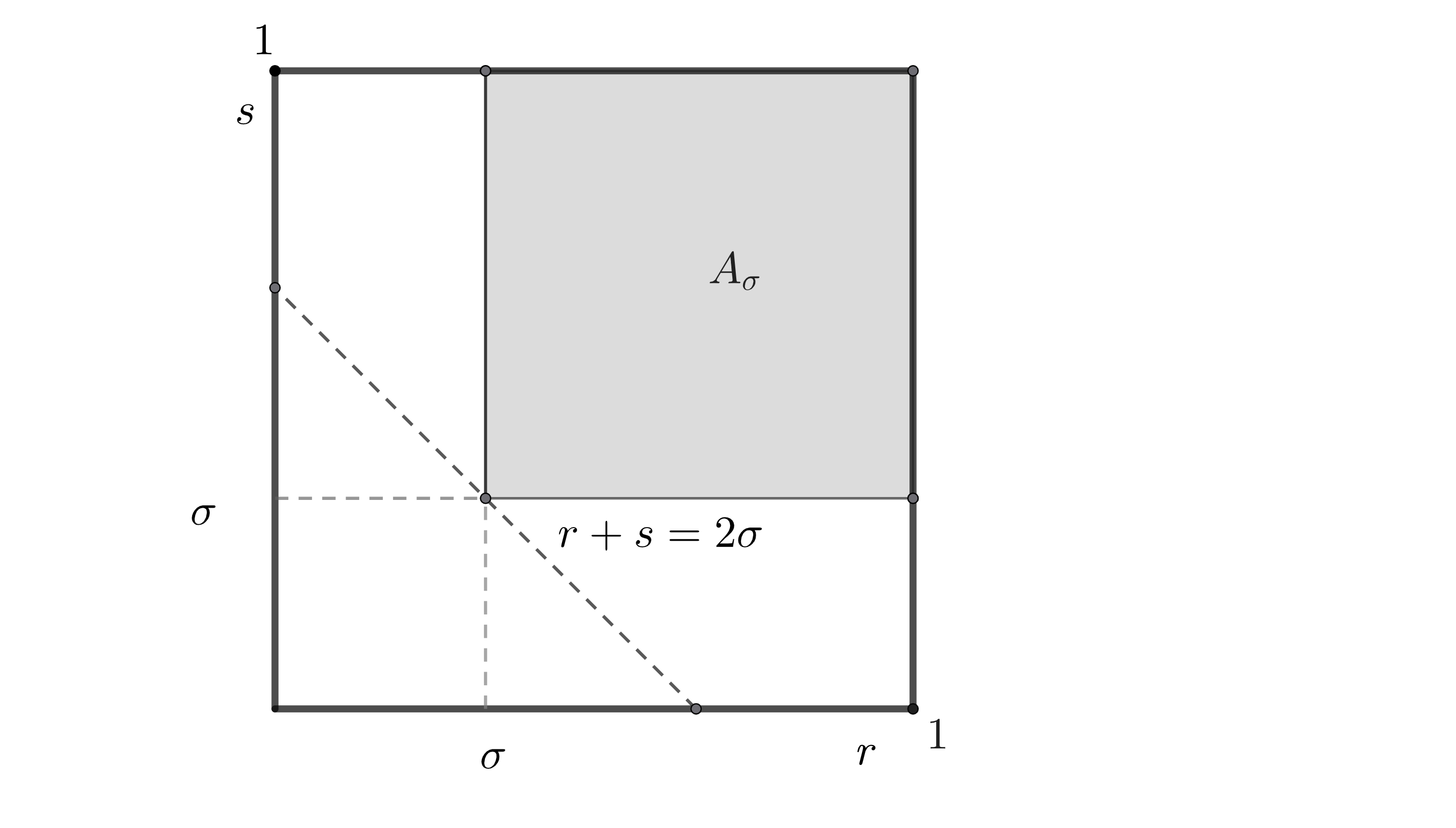}
	\end{center}
	\columnbreak
	\begin{center}
		\includegraphics[width=0.36\textwidth]{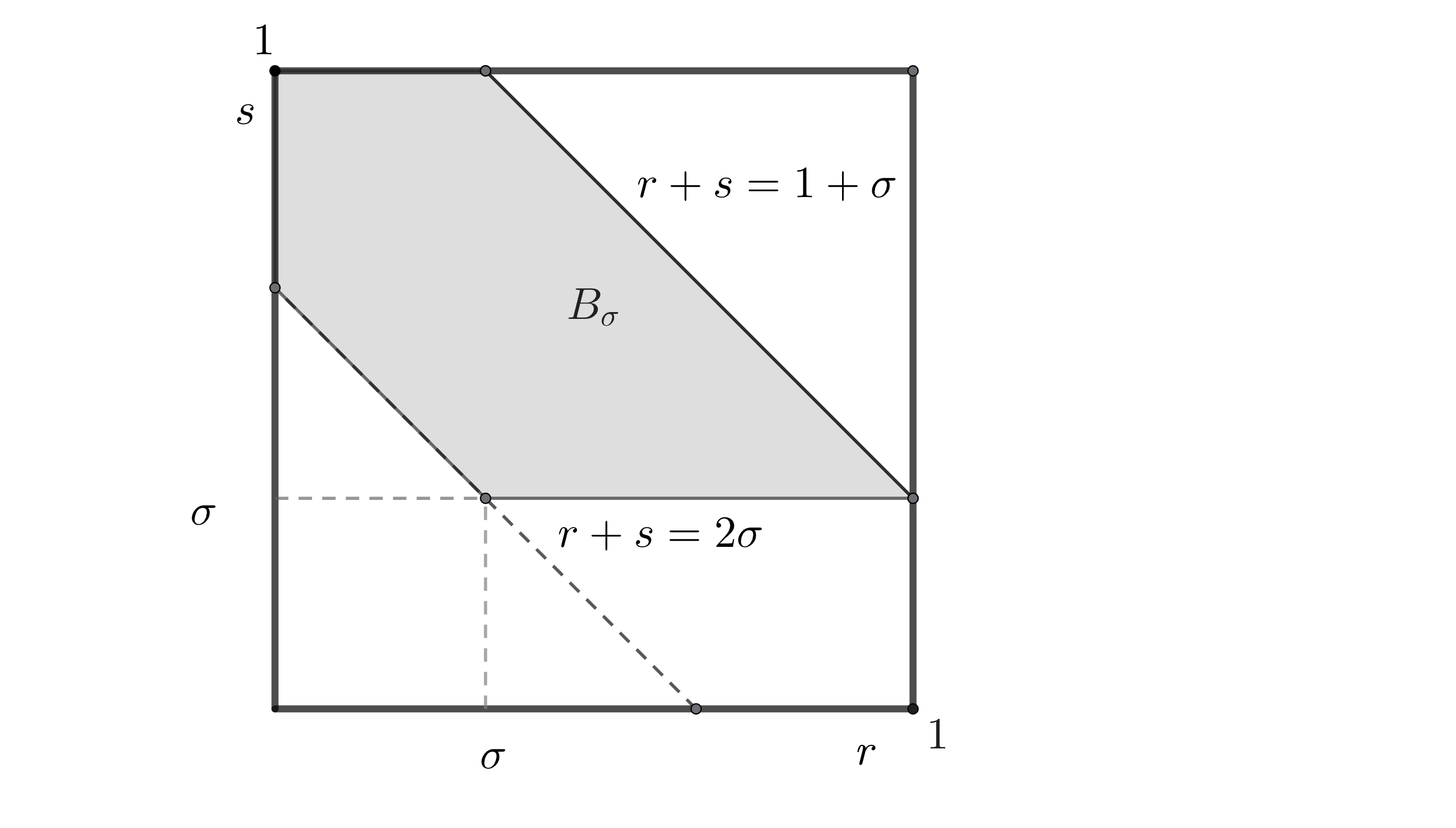}
	\end{center}
		\columnbreak
	\begin{center}
		\includegraphics[width=0.36\textwidth]{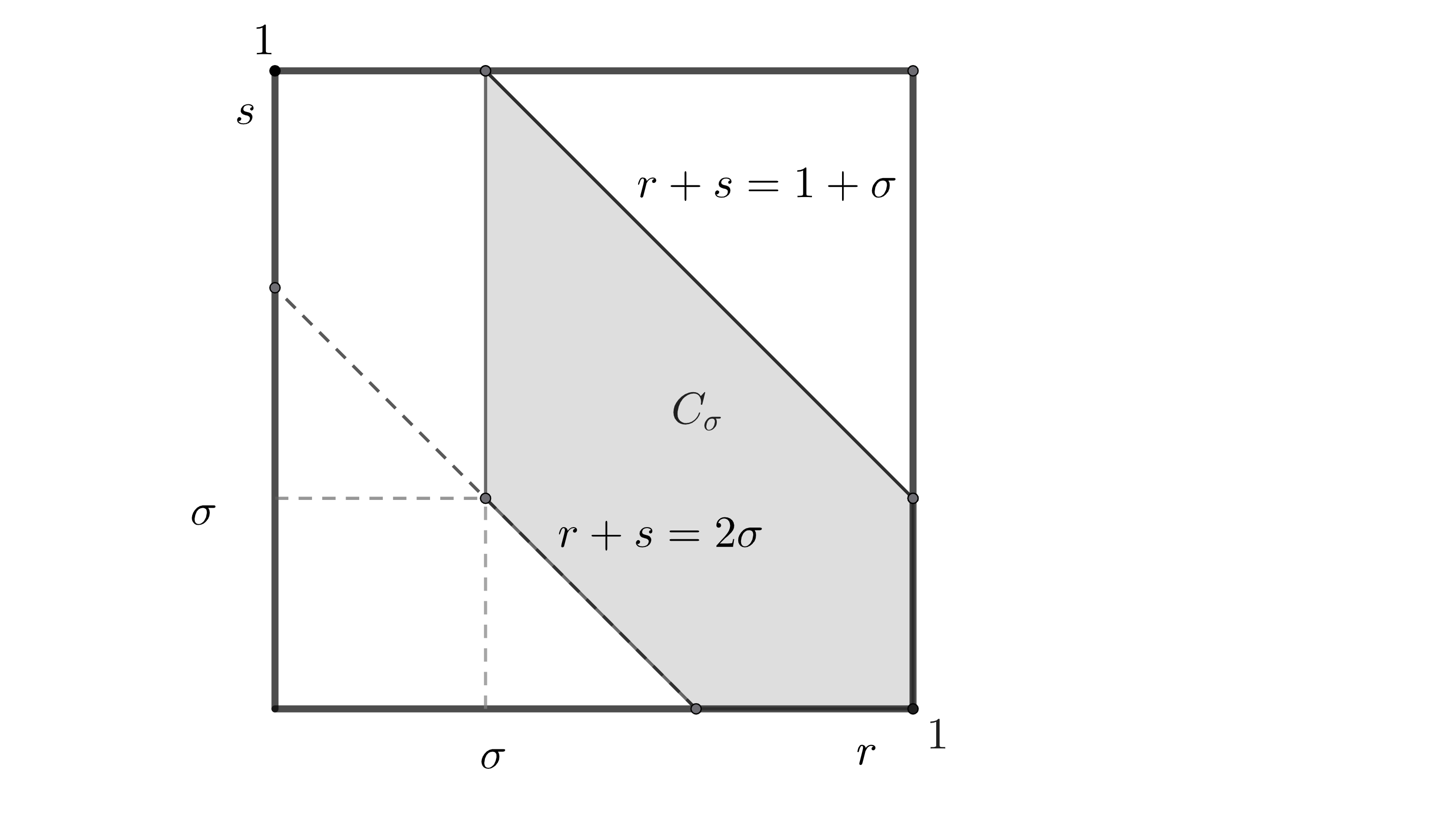}
	\end{center}
			\columnbreak
	\begin{center}
		\includegraphics[width=0.36\textwidth]{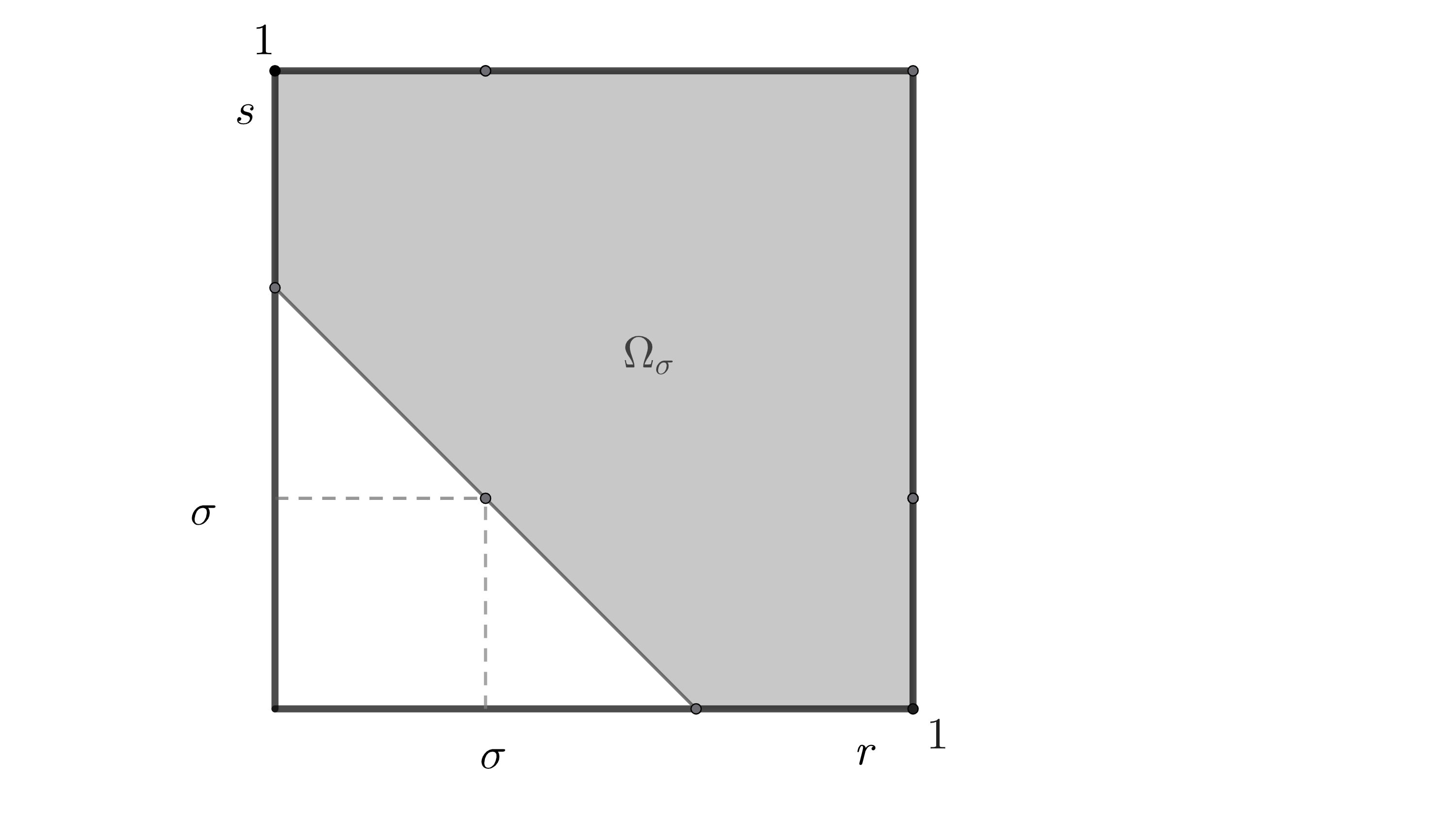}
	\end{center}
\end{multicols}
\caption{Pictures of $A_{\sigma}$, $B_{\sigma}$, $C_{\sigma}$ and $\Omega_\sigma$. The first row depicts the four regions when $\frac{1}{2}<\sigma<1$. The second row when $0<\sigma<\frac{1}{2}$. }\label{fig1}
\end{figure}
The next result provides the three basic estimates of the kernel of the operator $T^\gamma$.

\begin{proposition}\label{Tgamma<min Ietagamma}
Let $(X,d,\mu)$ be an $\eta$-Ahlfors regular space with $\eta>0$. Let $X^3,d^{(3)},\bigtriangleup_3$ and $\rho$ be defined as in the above section. Then,
\begin{enumerate}[label=\textit{(\thetheorem.\alph*)}, leftmargin=*, align=left, font=\color{black}]
\item \label{rhogamma<dgamma} the kernel $\rho^{-\gamma}(x,y,z)$ is bounded above by
\begin{equation*}
(2\kappa)^{\gamma}\min\left\{d(x,y)^{-\frac{\gamma}{2}}d(x,z)^{-\frac{\gamma}{2}}, d(x,y)^{-\frac{\gamma}{2}}d(y,z)^{-\frac{\gamma}{2}},d(x,z)^{-\frac{\gamma}{2}}d(z,y)^{-\frac{\gamma}{2}}\right\}
\end{equation*}
for every $(x,y,z)\in X^3$ and every $\gamma>0$;
\item \label{Tgamma<Ietagamma}for $f$ and $g$ two measurable nonnegative real functions and $0<\gamma<2\eta$, we have $T^{\gamma}(f,g)(x)$ is bounded above by
\begin{equation*}
(2\kappa)^\gamma \min\left\{\left(I_{\eta-\frac{\gamma}{2}}f(x)\right)\left(I_{\eta-\frac{\gamma}{2}}g(x)\right),  I_{\eta-\frac{\gamma}{2}}\left(fI_{\eta-\frac{\gamma}{2}}g\right)(x), I_{\eta-\frac{\gamma}{2}}\left(gI_{\eta-\frac{\gamma}{2}}f\right)(x)\right\}
\end{equation*}
for every $x\in X$, where $I_\alpha$ denotes the linear fractional integral operator of order $\alpha$ in the space $(X,d,\mu)$.
\end{enumerate}
\end{proposition}

\begin{proof}
Let us first check \ref{rhogamma<dgamma}. Given $(x,y,z)\in X^3$ and $\varepsilon>0$, there exists $u_\varepsilon \in X$ such that	
\begin{equation*}
\rho(x,y,z)+\varepsilon\geq d^{(3)}\left((x,y,z), (u_\varepsilon,u_\varepsilon,u_\varepsilon)\right)= \max \{d(x,u_\varepsilon),d(y,u_\varepsilon),d(z,u_\varepsilon)\}.
\end{equation*} 
Now from triangle inequality, $d(x,y)\leq \kappa\left(d(x,u_\varepsilon)+d(u_\varepsilon,y)\right)
\leq 	2\kappa\rho(x,y,z)+2\kappa\varepsilon$; $d(y,z)\leq \kappa\left(d(y,u_\varepsilon)+d(u_\varepsilon,z)\right)
\leq 	2\kappa\rho(x,y,z)+2\kappa\varepsilon$, and $	d(x,z)\leq \kappa\left(d(x,u_\varepsilon)+d(u_\varepsilon, z)\right)
\leq 	2\kappa\rho(x,y,z)+2\kappa\varepsilon$.
Hence 
\begin{equation*}
\rho^{-\frac{\gamma}{2}}(x,y,z)
\leq (2\kappa)^{-\frac{\gamma}{2}}\min\left\{d(x,y)^{-\frac{\gamma}{2}},d(x,z)^{-\frac{\gamma}{2}},d(y,z)^{-\frac{\gamma}{2}}\right\}.
\end{equation*}
So that the square of the left hand side of the above inequality is bounded by the product of any two of the three numbers on the right. In other words
\begin{equation*}
\rho^{-\gamma}(x,y,z)
\leq (2\kappa)^{\gamma}\min\left\{d(x,y)^{-\frac{\gamma}{2}}d(x,z)^{-\frac{\gamma}{2}}, d(x,y)^{-\frac{\gamma}{2}}d(y,z)^{-\frac{\gamma}{2}},d(x,z)^{-\frac{\gamma}{2}}d(z,y)^{-\frac{\gamma}{2}}\right\},
\end{equation*}
as desired.
	
Let us prove \ref{Tgamma<Ietagamma}. Given $f\geq0$ and $g\geq 0$, from \ref{rhogamma<dgamma} we have
\begin{align*}
T^{\gamma}(f,g)
&=\iint_{X\times X}\frac{f(y)g(z)}{\rho^{\gamma}(x,y,z)}d\mu(y)d\mu(z)\\
&\leq (2\kappa)^\gamma\left(\int_{y\in X}\frac{f(y)}{d^{\frac{\gamma}{2}}(x,y)}d\mu(y)\right)\left(\int_{z\in X}\frac{g(z)}{d^{\frac{\gamma}{2}}(x,z)}d\mu(z)\right)\\
&=(2\kappa)^\gamma\left(I_{\eta-\frac{\gamma}{2}}f\right)(x)\left(I_{\eta-\frac{\gamma}{2}}g\right)(x);
\end{align*}

\begin{align*}
T^{\gamma}(f,g)
&=\iint_{X\times X}\frac{f(y)g(z)}{\rho^{\gamma}(x,y,z)}d\mu(y)d\mu(z)\\
&\leq (2\kappa)^\gamma\int_{y\in X}\frac{f(y)}{d^{\frac{\gamma}{2}}(x,y)}\left(\int_{z\in X}\frac{g(z)}{d^{\frac{\gamma}{2}}(y,z)}d\mu(z)\right)d\mu(y)\\
&=(2\kappa)^\gamma I_{\eta-\frac{\gamma}{2}}\left(fI_{\eta-\frac{\gamma}{2}}g\right)(x);
\end{align*}
and
\begin{align*}
T^{\gamma}(f,g)
&=\iint_{X\times X}\frac{f(y)g(z)}{\rho^{\gamma}(x,y,z)}d\mu(y)d\mu(z)\\
&\leq (2\kappa)^\gamma \int_{z\in X}\frac{g(z)}{d^{\frac{\gamma}{2}}(x,z)}\left(\int_{y\in X}\frac{f(y)}{d^{\frac{\gamma}{2}}(z,y)}d\mu(y)\right)d\mu(z)\\
&=(2\kappa)^\gamma I_{\eta-\frac{\gamma}{2}}\left(gI_{\eta-\frac{\gamma}{2}}f\right)(x).
\end{align*}
\end{proof}

\begin{proposition}\label{|Ietagammafg|<c|f||g|}
Let $(X,d,\mu)$ be an $\eta$-Ahlfors regular space with $\eta>0$ and $\rho$ as before. Let $0<\gamma<2\eta$ be given. Set $\sigma=\frac{2\eta -\gamma}{2\eta}$ and $\Pi_\sigma=\{(r,s,t): r+s-t=2\sigma\}$.
Then, with the notation of Lemma~\ref{Omega=AuBuC} we have that
\begin{enumerate}[label=\textit{(\thetheorem.\alph*)}, leftmargin=*, align=left, font=\color{black}]
\item\label{desig 1} for every $(p_1,p_2;p_3)$ such that $\bigl(\frac{1}{p_1},\frac{1}{p_2},\frac{1}{p_3}\bigr) \in \Pi_{\sigma}$ and $\bigl(\frac{1}{p_1},\frac{1}{p_2}\bigr) \in A_{\sigma}$, there exists $C>0$ such that
\begin{equation*}
\left\|\left(I_{\eta-\frac{\gamma}{2}}f\right)\left(I_{\eta-\frac{\gamma}{2}}g\right)\right\|_{p_3}\leq C \|f\|_{p_1}\|g\|_{p_2}
\end{equation*}
for $f\in L^{p_1}(X,\mu)$ and $g\in L^{p_2}(X,\mu)$;
\item\label{desig 2 } for every $(p_1,p_2;p_3)$ such that $\bigl(\frac{1}{p_1},\frac{1}{p_2},\frac{1}{p_3}\bigr) \in \Pi_{\sigma}$ and $\bigl(\frac{1}{p_1},\frac{1}{p_2}\bigr) \in B_{\sigma}$, there exists $C>0$ such that
\begin{equation*}
\left\|I_{\eta-\frac{\gamma}{2}}\left(fI_{\eta-\frac{\gamma}{2}}g\right)\right\|_{p_3}\leq C \|f\|_{p_1}\|g\|_{p_2}
\end{equation*}
for $f\in L^{p_1}(X,\mu)$ and $g\in L^{p_2}(X,\mu)$;
\item \label{desig 3}
for every $(p_1,p_2;p_3)$ such that $\bigl(\frac{1}{p_1},\frac{1}{p_2},\frac{1}{p_3}\bigr) \in \Pi_{\sigma}$ and $\bigl(\frac{1}{p_1},\frac{1}{p_2}\bigr) \in C_{\sigma}$, there exists $C>0$ such that
\begin{equation*}
\left\|I_{\eta-\frac{\gamma}{2}}\left(gI_{\eta-\frac{\gamma}{2}}f\right)\right\|_{p_3}\leq C \|f\|_{p_1}\|g\|_{p_2}
\end{equation*}
for $f\in L^{p_1}(X,\mu)$ and $g\in L^{p_2}(X,\mu)$.
\end{enumerate}
\end{proposition}
\begin{proof}
Let us first prove \ref{desig 1}. Take	$\bigl(\frac{1}{p_1},\frac{1}{p_2},\frac{1}{p_3}\bigr)\in \Pi_\sigma$ with $\bigl(\frac{1}{p_1},\frac{1}{p_2}\bigr)\in A_{\sigma}$. So that
\begin{align}
	\frac{1}{p_1}+\frac{1}{p_2}-\frac{1}{p_3}&=2\sigma\label{hipot1}\\
	1>\frac{1}{p_1}&>\sigma\label{hipot2}\\
	1>\frac{1}{p_2}&>\sigma.\label{hipot3}
\end{align}
From the first inequality in \eqref{hipot2} we have that $p_1>1$, from the second, $\frac{1}{\pi_1}:=\frac{1}{p_1}-\sigma=\frac{1}{p_1}-\frac{\eta-\frac{\gamma}{2}}{\eta}>0$. From \eqref{hipot3}, the same argument shows that $\frac{1}{\pi_2}:=\frac{1}{p_2}-\sigma=\frac{1}{p_2}-\frac{\eta-\frac{\gamma}{2}}{\eta}>0$. Moreover, from \eqref{hipot1} we have 
\begin{equation*}
	\frac{1}{\pi_1}+\frac{1}{\pi_2}=\frac{1}{p_1}-\sigma+\frac{1}{p_2}-\sigma=\frac{1}{p_1}+\frac{1}{p_2}-2\sigma=\frac{1}{p_3}.
\end{equation*}
Hence, we can apply first H\"{o}lder inequality and then Hardy-Littlewood-Sobolev twice to obtain
\begin{align*}
\left\|\left(I_{\eta-\frac{\gamma}{2}}f\right)\left(I_{\eta-\frac{\gamma}{2}}g\right)\right\|_{p_3}&\leq\left\|I_{\eta-\frac{\gamma}{2}}f\right\|_{\pi_1}\left\|I_{\eta-\frac{\gamma}{2}}g\right\|_{\pi_2}\\
&\leq C\|f\|_{p_1}\|g\|_{p_2},
\end{align*}
as desired.

In order to prove \ref{desig 2 },  take $\bigl(\frac{1}{p_1},\frac{1}{p_2},\frac{1}{p_3}\bigr)\in \Pi_\sigma$ with  $\bigl(\frac{1}{p_1},\frac{1}{p_2}\bigr)\in B_\sigma$. Recall that the conditions are the following
\begin{align}
	\frac{1}{p_1}+\frac{1}{p_2}-\frac{1}{p_3}&=2\sigma\label{cond1}\\
	\frac{1}{p_1}+\frac{1}{p_2}&>2\sigma\label{cond2}\\
	\frac{1}{p_1}+\frac{1}{p_2}&<1+\sigma\label{cond3}\\
	\frac{1}{p_2}&>\sigma\label{cond4}
\end{align}
Let us proceed to introduce two new Lebesgue exponents given in terms of $\bigl(\frac{1}{p_1},\frac{1}{p_2},\frac{1}{p_3}\bigr)$. Set $\frac{1}{q_1}=\frac{1}{p_3}+\sigma$ and $\frac{1}{q_2}=\frac{1}{p_2}-\sigma$. Notice first that $1<q_1<\infty$. In fact, it is clear that $q_1<\infty$. Let us show that $q_1>1$. From \eqref{cond1} and \eqref{cond3} we see that $\frac{1}{q_1}=\frac{1}{p_3}+\sigma=\frac{1}{p_1}+\frac{1}{p_2}-2\sigma+\sigma=\frac{1}{p_1}+\frac{1}{p_2}-\sigma<1$. Let us check that $q_2>1$. From \eqref{cond4} we see that  $\frac{1}{q_2}=\frac{1}{p_2}-\sigma>0$. Now $\frac{1}{p_2}<\frac{1}{p_1}+\frac{1}{p_2}<1+\sigma$, from \eqref{cond3}, hence $0<\frac{1}{q_2}=\frac{1}{p_2}-\sigma<1$. Now we observe that $p_1$, $q_1$, and $q_2$ are related by the	H\"{o}lder identity through $\frac{1}{q_1}=\frac{1}{p_1}+\frac{1}{q_2}$. In fact $\frac{1}{q_1}-\frac{1}{q_2}=\frac{1}{p_3}+\sigma-\frac{1}{p_2}+\sigma=\frac{1}{p_3}-\frac{1}{p_2}+2\sigma=\frac{1}{p_1}$, form \eqref{cond1}. Now $p_3$ and $q_1$, on one hand, and $q_2$ and $p_2$, on the other, are related by the Hardy-Littlewood-Sobolev formulae for the exponents in the $\eta$ dimensional space $(X,d,\mu)$. To wit
\begin{equation*}
	\frac{1}{q_1}-\frac{\eta-\frac{\gamma}{2}}{\eta}=\frac{1}{q_1}-\sigma=\frac{1}{p_3},
\end{equation*}
and
\begin{equation*}
		\frac{1}{p_2}-\frac{\eta-\frac{\gamma}{2}}{\eta}=\frac{1}{p_2}-\sigma=\frac{1}{q_2}.
\end{equation*}
Hence the following estimates are obtained by applying twice the Hardy-Littlewood-Sobolev theorem and once H\"{o}lder inequality, with constant $C$ that may change from to line,
\begin{align*}
\left\|I_{\eta-\frac{\gamma}{2}}\left(fI_{\eta-\frac{\gamma}{2}}g\right)\right\|_{p_3}&\leq C \|fI_{\eta-\frac{\gamma}{2}}g\|_{q_1}\\
&\leq C\|f\|_{p_1}\|I_{\eta-\frac{\gamma}{2}}g\|_{q_2}\\
&\leq C\|f\|_{p_1}\|g\|_{p_2},
\end{align*}
as desired.

In order to prove \ref{desig 3} we only have to observe that when $\bigl(\frac{1}{p_1},\frac{1}{p_2},\frac{1}{p_3}\bigr)\in \Pi_{\sigma}$ and $\bigl(\frac{1}{p_1},\frac{1}{p_2}\bigr)\in C_\sigma$ condition \eqref{cond1}, \eqref{cond2} and \eqref{cond3} are symmetric with respect to the variables $r=\frac{1}{p_1}$ and $s=\frac{1}{p_2}$. For \eqref{cond4}, notice that now the condition becomes $\frac{1}{p_1}>\sigma$. Hence the result follows as in the case \ref{desig 2 }.
\end{proof}

Now the proof of Theorem~\ref{|T|p3leqC|f|p1|g|p2} follows readily from Lemma~\ref{Omega=AuBuC}, Proposition~\ref{Tgamma<min Ietagamma}  and Proposition~\ref{|Ietagammafg|<c|f||g|}.

	
%
\subsection*{Acknowledgements}
This work was supported in part by CONICET in Argentina, through grant PIET-R-2025 \#29820250100076CO.


\providecommand{\bysame}{\leavevmode\hbox to3em{\hrulefill}\thinspace}
\providecommand{\MR}{\relax\ifhmode\unskip\space\fi MR }
\providecommand{\MRhref}[2]{%
	\href{http://www.ams.org/mathscinet-getitem?mr=#1}{#2}
}
\providecommand{\href}[2]{#2}


\bigskip

\medskip

\noindent{\footnotesize
	\noindent\textit{Affiliation.\,}
	\textsc{Instituto de Matem\'{a}tica Aplicada del Litoral ``Dra. Eleonor Harboure'', UNL, CONICET.}

	\smallskip
	\noindent\textit{Address.\,}\textmd{IMAL, Streets F.~Leloir and A.P.~Calder\'on, CCT CONICET Santa Fe, Predio ``Alberto Ca\-ssa\-no'', Colectora Ruta Nac.~168 km~0, Paraje El Pozo, S3007ABA Santa Fe, Argentina.}

	\smallskip
	\noindent \textit{E-mail address.\, }\verb|haimar@santafe-conicet.gov.ar|; \\
	\hspace*{2.5cm} \verb|ivanagomez@santafe-conicet.gov.ar|;\\ 
	\hspace*{2.5cm} \verb|joaquintoledo@santafe-conicet.gov.ar| 
}

\end{document}